\documentclass[reqno]{amsproc}
\usepackage{amsmath}
\usepackage{amsfonts}
\usepackage{mathrsfs}
\usepackage[dvips]{graphicx}
\usepackage{amssymb}
\usepackage{amsbsy}
\usepackage{amsthm}
\usepackage{graphics}
\usepackage{color}
\usepackage{textcomp}

\makeatletter
\@namedef{subjclassname@2010}{%
\textup{2010} Mathematics Subject Classification}
\makeatother

\numberwithin{equation}{section}

\newtheorem{theorem}{Theorem}[section]

\newtheorem{prob}[theorem]{Problem}
\newtheorem{lemma}[theorem]{Lemma}

\theoremstyle{remark}

\newcommand*{\hh}{\mathcal{H}}

\newcommand*{\ascr}{\mathscr{A}}

\newcommand*{\natu}{\mathbb{N}}

\newcommand*{\comp}{\mathbb{C}}
\newcommand*{\borel}{\mathfrak{B}}

\newcommand*{\cbb}{\comp}

\newcommand*{\D}{\mathrm{d\hspace{.1ex}}}

\theoremstyle{definition}
\newtheorem{ex}[theorem]{Example}

\newcommand*{\Le}{\leqslant}
\newcommand*{\Ge}{\geqslant}

\newcommand*{\ogr}{\boldsymbol{B}}
\newcommand*{\kk}{\mathcal{K}}
\newcommand*{\rbb}{\mathbb{R}}
\newcommand*{\zbb}{\mathbb{Z}}
   \begin{document}

   \title[Subnormal $n$th roots of quasinormal operators are quasinormal]
{Subnormal $n$th roots of quasinormal
operators \\ are quasinormal}

   \author[P. Pietrzycki and J. Stochel]{Pawe{\l} Pietrzycki and  Jan Stochel}



   \address{Wydzia{\l} Matematyki i Informatyki, Uniwersytet
Jagiello\'{n}ski, ul. {\L}ojasiewicza 6, PL-30348
Krak\'{o}w}

   \email{pawel.pietrzycki@im.uj.edu.pl}

   \address{Wydzia{\l} Matematyki i Informatyki, Uniwersytet
Jagiello\'{n}ski, ul. {\L}ojasiewicza 6, PL-30348
Krak\'{o}w}

   \email{jan.stochel@im.uj.edu.pl}

   \begin{abstract}
In a recent paper \cite{Curto20}, R. E.
Curto, S. H. Lee and J. Yoon asked the
following question: \textit{Let $A$ be a
subnormal operator, and assume that $A^2$
is quasinormal. Does it follow that $A$
is quasinormal?} In this paper, we answer
that question in the affirmative. In
fact, we prove a more general result that
subnormal $n$th roots of quasinormal
operators are quasinormal.
   \end{abstract}

   \maketitle

   \section{Introduction}
An operator\footnote{In this paper, by
``an operator'' we mean ``a bounded
linear operator''.} $A$ on a (complex)
Hilbert space $\hh$ is said to be {\em
subnormal} if it is (unitarily equivalent
to) the restriction of a normal operator
to its invariant (closed vector)
subspace. In turn, $A$ is called
\textit{quasinormal} if
$A(A^*A)=(A^*A)A$, or equivalently, if
and only if $U|A| = |A|U$, where $A =
U|A|$ is the polar decomposition of $A$
(see \cite[Theorem~7.20]{Weid80}). The
classes of subnormal and quasinormal
operators were introduced by P. Halmos in
\cite{hal50} and by A. Brown in
\cite{brow53}, respectively. It is
well-known that quasinormal operators are
subnormal but not conversely (see
\cite[Problem~195]{hal82}). For more
information on subnormal and quasinormal
operators we refer the reader to
\cite{hal82,con91}.

In a recent paper \cite{Curto20}, R. E.
Curto, S. H. Lee and J. Yoon, partially
motivated by the results of their
previous articles \cite{Curto05,Curto07},
asked the following question
   \begin{prob}[{\cite[Problem~1.1]{Curto20}}] \label{prob}
Let $A$ be a subnormal operator, and
assume that $A^2$ is quasinormal. Does it
follow that $A$ is quasinormal\/{\em ?}
   \end{prob}
They proved that a left invertible
subnormal operator $A$ whose square
$A^2$ is quasinormal, must be
quasinormal (see
\cite[Theorem~2.4]{Curto20}). It
remains an open question as to whether
this is true in general, without
assuming left invertibility. In this
paper, we show that Problem~\ref{prob}
has an affirmative answer. In fact, we
prove the following more general result
(see also Theorem~\ref{maintw2} for an
even more general statement).
   \begin{theorem}\label{maintw}
Let $A$ be a subnormal operator on a
Hilbert space $\hh$ and $n$ be an
integer greater than $1$. Assume that
$A^n$ is quasinormal. Then $A$ is
quasinormal.
   \end{theorem}
In Section~\ref{Sec.3} we give two
proofs of Theorem~\ref{maintw}. In both
cases we use Embry's characterization
of quasinormal operators which, for the
reader's convenience, is stated
explicitly below. In fact, this
characterization takes two equivalent
forms described by the conditions (ii)
and (iii) of Theorem~\ref{embry}. In
the first proof of Theorem~\ref{maintw}
we exploit the condition (ii), while in
the second the condition (iii). This is
because we use two completely different
techniques for proving
Theorem~\ref{maintw}. The first one
appeals to the theory of operator
monotone functions, in particular to
Hansen's inequality. The other relies
on the theory of (scalar and operator)
moment problems; its origin goes back
to the celebrated Embry's
characterization of subnormal operators
expressed in terms of the Stieltjes
operator moment problem \cite{Embry73},
later on developed by Lambert
\cite{Lam76} and Agler
\cite[Theorem~3.1]{Ag85}. In the next
section, we give basic information on
these techniques.

Similar questions to that in
Problem~\ref{prob} concerning square
roots (or more generally $n$th roots)
in selected classes of operators have
been studied since at least the early
$50$'s (see e.g.,
\cite{hal53,hal54,put57,wog85,con87}).
In particular, it is known that the
hyponormal $n$th roots of normal
operators are normal (see
\cite[Theorem~5]{sta62}). However, if
$A$ is a hyponormal operator and $A^n$
is subnormal, then $A$ doesn't have to
be subnormal (see \cite[pp.\
378/379]{sta66}). It is also worth
mentioning that there are subnormal (or
even isometric) operators which have no
square roots (see
\cite[Problem~145]{hal82}; see also
\cite{hal53}).

We now state Embry's characterization
of quasinormal operators (the
``moreover'' part of
Theorem~\ref{embry} follows from the
observation that due to \eqref{kul-3},
$E$ is the spectral measure of $A^*A$).
   \begin{theorem}[Embry's
characterization {\cite[page
63]{Embry73}}] \label{embry} Let $A$ be
an operator on $\mathcal{H}$. Then the
following conditions are equivalent{\em
:}
   \begin{itemize}
   \item[(i)] $A$ is quasinormal,
   \end{itemize}
   \begin{itemize}
   \item[(ii)] $A^{*k}A^{k}=(A^*A)^k$ for $k =0,1,2,\ldots$,
   \item[(iii)] there exists a Borel spectral measure $E$ on
$\rbb_{+}$ such that
   \begin{equation} \label{kul-3}
A^{*k}A^k = \int_{\rbb_{+}} x^k E(\D
x), \quad k =0,1,2,\ldots.
   \end{equation}
   \end{itemize}
Moreover, the spectral measure $E$ in
{\em (iii)} is unique and $E((\|A\|^2,
\infty))=0$.
   \end{theorem}
(We refer the reader to
\cite[Theorem~3.6]{jabl14} for the
version of the above theorem for
operators that are not necessarily
bounded; cf.\ also \cite{uch93}). The
condition (ii) of Theorem~\ref{embry}
leads to the following question.
   \begin{prob}  \label{prob3}
For what subsets $S$ of
$\{1,2,3,\ldots\}$ does the system of
equations
   \begin{equation}\label{xukl}
A^{*k}A^k=(A^*A)^k, \quad k\in S,
   \end{equation}
imply the quasinormality of $A$?
   \end{prob}
This problem, to some extent related to
the theory of operator monotone and
operator convex functions, has been
studied by several authors (see
\cite{uch93,uch01,Jib10,jabl14,pp16,pp18}).
In particular, $A$ is quasinormal if
any of the following conditions~holds:
   \begin{itemize}
   \item  $A$ is compact (hyponormal,
a unilateral or bilateral weighted shift)
and satisfies \eqref{xukl} with
$S=\{n\}$, where $n$ is a fixed integer
greater than $1$ (see \cite[p.\
198]{uch93}, \cite[Theorem~5.3]{pp18} and
\cite[Theorem~3.3]{pp16}),
   \item $A$ is log-hyponormal and satisfies
\eqref{xukl} with $S=\{n\}$, where $n$ is
a fixed integer greater than $2$ (see
\cite[Theorem~4.1]{uch01}),
   \item $A$ satisfies \eqref{xukl}
with $S=\{m,n,p,m+p,n+p\}$, where $m,n,p$
are fixed positive integers such that
$m<n$ (see \cite[Theorem~3.11]{pp18}; see
also \cite[Theorem~2.1]{uch93} for the
case $p=1$; in particular, taking $m=p=1$
and $n=2$ covers the case $S=\{2,3\}$
considered much later in
\cite[Proposition~13]{Jib10}).
   \end{itemize}
On the other hand, for every integer
$n\Ge 2$, there exists an operator $A$
such that
  \begin{equation}\label{qqq}
\text{$A^{*n}A^{n}=(A^{*} A)^n$ and
$A^{*k}A^{k} \neq (A^{*} A)^k$ for all
$k\in\{2,3,4,\ldots\}\setminus \{n\}$.}
  \end{equation}
Examples of such operators are to be
found in the classes of weighted shifts
on directed trees and composition
operators on $L^2$-spaces (see
\cite[Example~5.5]{jabl14} and
\cite[Theorem~4.3]{pp16}).

It turns out that the operators
satisfying \eqref{xukl} with
$S=\{\kappa\}$, where $\kappa$ is a
fixed integer greater than $1$, can
successfully replace quasinormal
operators in the predecessor of the
implication in Problem~\ref{prob} (see
Theorem~\ref{maintw2}). The adaptation
of both techniques used in the proof of
Theorem~\ref{maintw} to the proof of
Theorem~\ref{maintw2} resulted in
finding a new criterion for a
semi-spectral measure to be spectral,
expressed in terms of its two
``moments'' (see
Theorem~\ref{maintw4}).
   \section{Preliminaries}
In this paper, we use the following
notation. The fields of real and
complex numbers are denoted by $\rbb$
and $\mathbb{C}$, respectively. The
symbols $\zbb_{+}$, $\mathbb{N}$ and
$\rbb_+$ stand for the sets of
nonnegative integers, positive integers
and nonnegative real numbers,
respectively. We write $\borel(X)$ for
the $\sigma$-algebra of all Borel
subsets of a topological Hausdorff
space $X$.

A sequence
$\{\gamma_n\}_{n=0}^{\infty}$ of real
numbers is said to be a {\em Stieltjes
moment sequence} if there exists a
positive Borel measure $\mu$ on
$\rbb_+$ such that
   \begin{align} \label{hamb}
\gamma_n = \int_{\rbb_+} t^n d \mu(t),
\quad n\in \zbb_{+}.
   \end{align}
A positive Borel measure $\mu$ on
$\rbb_+$ satisfying \eqref{hamb} is
called a {\em representing measure} of
$\{\gamma_n\}_{n=0}^{\infty}$. If
$\{\gamma_n\}_{n=0}^{\infty}$ is a
Stieltjes moment sequence which has a
unique representing measure, then we
say that $\{\gamma_n\}_{n=0}^{\infty}$
is {\em determinate}. It is well known
that if a Stieltjes moment sequence has
a representing measure with compact
support, then it is determinate. The
reader is referred to \cite{B-C-R} for
comprehensive information regarding the
Stieltjes moment problem.

Let $\hh$ be a Hilbert space. Denote by
$\ogr(\hh)$ the $C^*$-algebra of all
bounded linear operators on $\hh$, and by
$I_{\hh}$ the identity operator on $\hh$.
As usual, $A^*$ stands for the adjoint of
$A\in \ogr(\hh)$. We say that an operator
$A\in \ogr(\hh)$ is
   \begin{itemize}
   \item \textit{positive} if   $\langle Ah,h\rangle \Ge 0$
for all $h\in \hh$,
   \item an \textit{orthogonal projection} if $A=A^*$ and $A=A^2$,
   \item \textit{selfadjoint} if $A=A^*$,
   \item \textit{normal} if $A^*A=AA^*$,
   \item \textit{quasinormal} if $A(A^*A)=(A^*A)A$,
   \item \textit{subnormal} if it is (unitarily equivalent to) the restriction
of a normal operator to its invariant subspace.
   \end{itemize}

Let $\ascr$ be a $\sigma$-algebra of
subsets of a set $X$ and let $F\colon
\ascr \to \ogr(\hh)$ be a semispectral
measure, that is $\langle F (\cdot)f,
f\rangle$ is a positive measure for every
$f \in \hh$, and $F (X) = I_\hh$. Denote
by $L^1(F)$ the vector space of all
$\ascr$-measurable functions $f\colon X
\to \cbb$ such that $\int_{X} |f(x)|
\langle F(\D x)h, h\rangle < \infty$ for
all $h\in \hh$. Then for every $f\in
L^1(F)$, there exists a unique operator
$\int_X f \D F \in \ogr(\hh)$ such that
(see e.g., \cite[Appendix]{Sto92})
   \begin{align} \label{form-ua}
\Big\langle\int_X f \D F h, h\Big\rangle
= \int_X f(x) \langle F(\D x)h, h\rangle,
\quad h\in\hh.
   \end{align}
If $F$ is a spectral measure, that is
$F(\varDelta)$ is an orthogonal
projection for every $\varDelta \in
\ascr$, then $\int_X f \D F$ coincides
with the usual spectral integral. In
particular, if $F$ is the spectral
measure of a normal operator $A$, then we
write $f(A)=\int_{\mathbb{C}} f \D F$ for
any $F$-essentially bounded Borel
function $f\colon \cbb \to \cbb$; the map
$f \mapsto f(A)$ is called the Stone-von
Neumann functional calculus. We refer the
reader to \cite{Rud73,Weid80,Sch12} for
the necessary information on spectral
integrals, including the spectral theorem
for normal operators and the Stone-von
Neumann functional calculus, which we
will need in this paper.

The following fact can be deuced from the
spectral theorem
\cite[Theorem~12.23]{Rud73} by applying
the Stone-von Neumann functional calculus
(cf.\ also \cite[Theorem~7.20]{Weid80}).
   \begin{theorem}\label{rsn} If $p$ is a positive number,
then the commutants of a positive
operator and its $p$th power coincide.
   \end{theorem}
Let $J\subset \rbb$ be an interval
(which may be open, half-open, or
closed; finite or infinite). A
continuous function $f \colon J
\rightarrow \rbb$ is said to be
\textit{operator monotone} if $f(A)\Le
f(B)$ for any two selfadjoint operators
$A,B\in\ogr(\hh)$ such that $A\Le B$
and the spectra of $A$ and $B$ are
contained in $J$. In 1934 K. L\"owner
\cite{Lo34} proved that a continuous
function defined on an open interval is
operator monotone if and only if it has
an analytic continuation to the complex
upper half-plane which is a Pick
function (cf.\ \cite{Dono74,Han13}).
Operator monotone functions constitute
an important class of real-valued
functions that has a variety of
applications in other branches of
mathematics. What is more, operator
monotone functions have integral
representations with respect to
suitable positive Borel measures. In
particular, a continuous function
\mbox{$f \colon (0, \infty) \rightarrow
\rbb$} is operator monotone if and only
if there exists a positive Borel
measure $\nu$ on $[0,\infty)$ such that
$\int_0^\infty \frac{1}{1+\lambda^2}\D
\nu(\lambda)<\infty$ and
   \begin{equation} \label{repbol}
f(t)=\alpha +\beta t +\int_0^\infty
\Big(\frac{\lambda}{1+\lambda^2} -
\frac{1}{t+\lambda}\Big) \D \nu(\lambda),
\quad t \in (0,\infty),
   \end{equation}
where $\alpha\in \rbb$ and $\beta \in
\rbb_+$ (see \cite[Theorem~5.2]{Han13}
or \cite[p.\ 144]{Bha97}). Below, we
give an important example of a function
which is operator monotone.
   \begin{ex}\label{ex}
For $p\in (0,1)$, the function $f\colon
[0,\infty)\ni t\rightarrow t^p\in\rbb$
is operator monotone and has the
following integral representation (see
\cite[Exercise~V.1.10(iii)]{Bha97} or
\cite[Exercise~V.4.20]{Bha97})
   \begin{equation*}
t^p=\frac{\sin p\pi}{\pi} \int_0^\infty
\frac{t\lambda^{p-1}}{t+\lambda} \D
\lambda, \quad t \in [0,\infty).
   \end{equation*}
   \end{ex}
The fact that the function in Example
\ref{ex} is operator monotone is known as
the L\"owner-Heinz inequality.
   \begin{theorem}[L\"owner-Heinz inequality
\cite{He51,Lo34}] \label{lohe} If $A,B\in
\ogr(\hh)$ are positive operators such
that $B \Le A$ and $p\in [0,1]$, then
$B^p \Le A^p$.
   \end{theorem}
Another inequality related to operator
monotone functions that is needed in
this paper is the Hansen inequality
\cite{Han80}. In
\cite[Lemma~2.2]{uch93}, M. Uchiyama
gave a necessary and sufficient
condition for equality to hold in the
Hansen inequality when the external
factor is a nontrivial orthogonal
projection (see the ``moreover'' part
of Theorem~\ref{hans} below). The key
ingredient of the proof of
\cite[Lemma~2.2]{uch93} is the integral
representation \eqref{repbol} of
operator monotone functions (to be more
precise, a version of \eqref{repbol}
for the interval $J=\rbb_+$ as in
\cite[pp.\ 144-145]{Bha97}). In the
original formulation of this lemma,
Uchiyama assumed that the underlying
Hilbert space $\hh$ is separable. This
assumption can be dropped due the fact
that for each vector $h\in \hh$ the
smallest closed vector subspace of
$\hh$ reducing both $A$ and $T$ and
containing $h$ is separable.
   \begin{theorem}[\cite{Han80,uch93}]
\label{hans} Let $A\in \ogr(\hh)$ be a
positive operator, $T \in \ogr(\hh)$ be
a contraction and $f\colon
[0,\infty)\rightarrow \rbb$ be a
continuous operator monotone function
such that $f(0)\Ge 0$. Then
   \begin{equation} \label{Han-inq}
T^*f(A)T \Le f(T^*AT).
   \end{equation}
Moreover, if $f$ is not an affine
function and $T$ is an orthogonal
projection such that $T \neq I_{\hh}$,
then equality holds in \eqref{Han-inq}
if and only if $TA=AT$ and $f(0)=0$.
   \end{theorem}
For more information on operator monotone
functions the reader is referred to
\cite{Lo34,Dono74,Han80,Bha97,Han13,Sim19}.
   \section{\label{Sec.3}Proofs of the main theorem}
In this section, we will give two
proofs of Theorem \ref{maintw}. We
start with a proof that uses the
technique of operator monotone
functions, including Uchiyama's
condition guaranteeing equality in
Hansen's inequality.
   \begin{proof}[First proof of Theorem \ref{maintw}]
Let $N\in \ogr(\kk)$ be a normal
extension of $A$. There is no loss of
generality in assuming that $\hh^\perp :=
\kk \ominus \hh \neq \{0\}$. Then $N$ has
the $2 \times 2$ matrix representation
   \begin{align} \label{brep}
N = \begin{bmatrix} A & B \\ 0 & C
\end{bmatrix},
   \end{align}
with respect to the orthogonal
decomposition $\kk = \hh \oplus
\hh^\perp$, where $B$ is a bounded linear
operator from $\hh^\perp$ to $\hh$ and $C
\in \ogr(\hh^\perp)$ (see \cite[p.\
39]{con91}). The orthogonal projection
$P\in\ogr(\kk)$ of $\kk$ onto $\hh$ has
the $2 \times 2$ matrix representation
   \begin{equation*}
P = \begin{bmatrix} I_\hh & 0\\ 0 & 0
\end{bmatrix}.
   \end{equation*}
Clearly $P\neq I_{\kk}$. It is now a
routine matter to verify that
   \begin{equation}\label{uch}
P(N^{*}N)^kP = PN^{*k}N^kP
\overset{\eqref{brep}} =
\begin{bmatrix} A^{*k}A^k   & 0 \\ 0 & 0
\end{bmatrix}, \quad k\in \mathbb Z_+.
   \end{equation}
Since $A^n$ is quasinormal,
Theorem~\ref{embry}(ii) yields
   \begin{equation}\label{em}
(A^n)^{*k}(A^n)^{k}=[(A^n)^*(A^n)]^{k},\quad
k\in \mathbb Z_+.
   \end{equation}
Fix any integer $\kappa \Ge 2$. Using
the Stone-von Neumann functional
calculus, we obtain
   \allowdisplaybreaks
   \begin{align} \notag
P(N^{*}N)^nP & \overset{\eqref{uch}}=
\begin{bmatrix} A^{*n}A^n & 0 \\ 0 & 0
\end{bmatrix}
   \\  \notag
& \overset{\eqref{em}}=
\begin{bmatrix} (A^{*\kappa n}A^{\kappa n})^\frac{1}{\kappa}
& 0 \\ 0 & 0 \end{bmatrix}
   \\ \label{main}
& \overset{\eqref{uch}}=
(P(N^{*}N)^{\kappa
n}P)^{\frac{1}{\kappa}}.
   \end{align}
Let $f\colon [0,\infty)\to \rbb$ be the
function given by
$f(x)=x^{\frac{1}{\kappa}}$ for $x\in
[0,\infty)$. It follows from Theorem
\ref{lohe} (or Example \ref{ex}) that
$f$ is an operator monotone function.
Using the Stone-von Neumann functional
calculus again and \eqref{main}, we get
   \begin{equation*}
Pf((N^{*}N)^{\kappa
n})P=f(P(N^{*}N)^{\kappa n}P).
   \end{equation*}
We conclude from Theorem~\ref{hans}
that $P$ commutes with
$(N^{*}N)^{\kappa n}$. By
Theorem~\ref{rsn}, $P$ commutes with
$N^{*}N$. This in turn implies that
   \allowdisplaybreaks
   \begin{align*}
   \begin{bmatrix} A^{*k}A^k & 0 \\ 0 & 0
\end{bmatrix} & \overset{\eqref{uch}} = P(N^{*}N)^k P
   \\
& \hspace{1ex}= (P(N^{*}N)P)^k
   \\
& \overset{\eqref{uch}}=
\begin{bmatrix} (A^{*}A)^k & 0 \\ 0 & 0
\end{bmatrix}, \quad k\in \natu.
   \end{align*}
Hence $A^{*k}A^k=(A^{*}A)^k$ for all
$k\in \mathbb Z_+$. Combined with
Theorem~\ref{embry}, this implies that
$A$ is quasinormal.
   \end{proof}
We now turn to the second proof of the
main theorem. This time the proof is
based on the technique of (semi)
spectral integrals and the Stieltjes
moment problem.
   \begin{proof}[Second proof of Theorem \ref{maintw}]
Let $N\in \ogr{(\kk)}$ be a minimal
normal extension of $A$, $G_N\colon
\borel(\cbb) \to \ogr(\kk)$ be the
spectral measure of $N$ and $P\in
\ogr(\kk)$ be the orthogonal projection
of $\kk$ onto $\hh$. Then the map
$\varTheta\colon \borel(\cbb) \to
\ogr(\hh)$ defined~by
   \begin{equation*}
\varTheta(\varDelta) =
PG_N(\varDelta)|_{\hh}, \quad \varDelta
\in\borel(\cbb),
   \end{equation*}
is a semispectral measure\footnote{By
\cite[Proposition~5]{Ju-St08} and
\cite[Proposition~II.2.5]{con91}, the
definition of $\varTheta$ does not
depend on the choice of $N$.} such that
$\varTheta(\{z\in \cbb\colon |z| >
\|N\|\})=0$. Since $A^k=N^k|_{\hh}$ and
$A^{*k}=PN^{*k}|_{\hh}$ for all $k\in
\zbb_+$, the Stone-von Neumann
functional calculus gives
   \begin{equation} \label{add-2}
A^{*k}A^k = PN^{*k}N^k|_{\hh} =
\int_{\cbb} |z|^{2k} \varTheta(\D
z),\quad k\in \zbb_+.
   \end{equation}
Let $F\colon \borel(\rbb_+) \to
\ogr(\hh)$ be the semispectral measure
defined by
   \begin{equation*}
F(\varDelta) =
\varTheta(\phi^{-1}(\varDelta)),\quad
\varDelta \in \borel(\rbb_+),
   \end{equation*}
where $\phi\colon \comp\to \rbb_+$ is
given by $\phi(z)=|z|^2$ for $z\in
\comp$. Then $F((\|N\|^2, \infty))=0$.
By \eqref{form-ua}, \eqref{add-2} and
the measure transport theorem (cf.\
\cite[Theorem~1.6.12]{Ash00}), we~have
   \begin{equation} \label{semi}
A^{*k}A^k=\int_{\rbb_+}x^k{F}(\D x),
\quad k\in \zbb_+.
   \end{equation}
Since $A^n$ is quasinormal,
Theorem~\ref{embry} implies that
   \begin{equation}\label{measure}
A^{*nk}A^{nk}=\int_{\rbb_+}x^kE_n(\D
x), \quad k\in \zbb_+,
   \end{equation}
where $E_n\colon \borel(\rbb_+) \to
\ogr(\hh)$ is a spectral measure. Let
$\psi_n\colon \rbb_+\to\rbb_+$ be given
by $\psi_n(x)= x^n$ for $x\in \rbb_+$.
Applying the measure transport theorem
and \eqref{form-ua}, we get
   \allowdisplaybreaks
   \begin{align} \notag
\int_{\rbb_+}x^kE_n(dx) &
\overset{\eqref{measure}} =
A^{*nk}A^{nk}
   \\  \notag
& \overset{\eqref{semi}} =
\int_{\rbb_+}[\psi_n(x)]^k{F}(\D x)
   \\ \label{Sti-det}
& \hspace{1ex}=\int_{\rbb_+}x^k(F\circ
\psi_n^{-1})(\D x), \quad k\in \zbb_+,
   \end{align}
where $F\circ \psi_n^{-1}\colon
\borel(\rbb_+) \to \ogr(\hh)$ is the
semispectral measure defined by
   \begin{align} \label{num-er}
(F\circ \psi_n^{-1})(\varDelta) =
F(\psi_n^{-1}(\varDelta)),\quad
\varDelta \in \borel(\rbb_+).
   \end{align}
Clearly, $(F\circ
\psi_n^{-1})((\|N\|^{2n},\infty))=0$.
Using the well-known fact that a
Stieltjes moment sequence having a
representing measure with compact support
is determinate (see e.g.,
\cite[(1.4)]{BJJS18}), we deduce from
\eqref{form-ua} and \eqref{Sti-det} that
   \begin{equation} \label{concl-de}
E_n (\varDelta)= ({F}\circ
\psi_n^{-1})(\varDelta),\quad \varDelta
\in \borel(\rbb_+).
   \end{equation}
Since $E_n$ is a spectral measure and
the map $\borel(\rbb_+)\ni \varDelta
\to \psi_n^{-1}(\varDelta) \in
\borel(\rbb_+)$ is bijective, we
conclude form \eqref{concl-de} that $F$
is a spectral measure. Combined with
\eqref{semi} and Theorem~\ref{embry},
this implies that $A$ is quasinormal.
   \end{proof}
   \section{A generalization and related matter}
Following the discussion in
Introduction concerning the reduced
Embry's characterization of
quasinormality (see
Problem~\ref{prob3}), we now deal with
operators satisfying the following
identity
   \begin{equation}\label{xukl-5}
T^{*\kappa}T^{\kappa}=(T^*T)^{\kappa},
   \end{equation}
where $\kappa$ is a fixed integer
greater than $1$. By
Theorem~\ref{embry}, any quasinormal
operator $T$ satisfies the single
equation \eqref{xukl-5}, but not
conversely (see \eqref{qqq}). It is
worth mentioning that operators
satisfying \eqref{xukl-5} with
$\kappa=2$ were investigated in
\cite{Jib10}; they form a subclass of
paranormal operators (see
\cite[Corollary~4.2]{jabl14}; see also
\cite[Theorem~3.5.1.1]{Fur01} for a
similar result stated for a wider
collection of the so-called class A
operators).

An inspection of the first proof of
Theorem~\ref{maintw} (see
Section~\ref{Sec.3}) reveals that the
following more general result is true.
   \begin{theorem} \label{maintw2}
Let $A$ be a subnormal operator on a
Hilbert space $\hh$ and $n, \kappa$ be
integers greater than $1$. Assume that
$T=A^n$ satisfies the single equation
\eqref{xukl-5}. Then $A$ is
quasinormal.
   \end{theorem}
Also, the second proof of
Theorem~\ref{maintw} can be adapted to
prove Theorem~\ref{maintw2}. Namely, it
suffices to apply Theorem~\ref{maintw4}
below, which is of independent
interest, to the operator
$T=A^{*n}A^{n}$, the semispectral
measure $F\circ \psi_n^{-1}$ and the
exponents $\alpha=1$ and $\beta=\kappa$
(see \eqref{Sti-det}).
   \begin{theorem} \label{maintw4}
Let $T\in \ogr(\hh)$ be a positive
operator on a Hilbert space $\hh$ and
$\alpha,\beta$ be two distinct positive
numbers. Assume that $F\colon
\borel(\rbb_+) \to \ogr(\hh)$ is a
semispectral measure with compact
support. Then the following conditions
are~equivalent{\em :}
   \begin{enumerate}
   \item[(i)] $F$ is a spectral
measure,
   \item[(ii)] $T^n = \int_{\rbb_+} x^n F(\D
x)$ for all $n\in \zbb_{+}$,
   \item[(iii)] $T^p = \int_{\rbb_+} x^p F(\D
x)$ for all $p \in \rbb_+$,
   \item[(iv)] $T^p =\int_{\rbb_+}
x^p F(\D x)$ for $p= \alpha, \beta$.
   \end{enumerate}
   \end{theorem}
Before we turn to the proof of
Theorem~\ref{maintw4}, we will give a
necessary and sufficient condition for
a semispectral measure to be spectral.
The proof of Lemma~\ref{maintw3} below
combines the two techniques used in
this article.
   \begin{lemma} \label{maintw3}
Let $F\colon \borel(\rbb_+) \to
\ogr(\hh)$ be a semispectral measure
with compact support and $p$ be a
positive number other than $1$. Assume
that
   \begin{align} \label{kap-lok}
\Big(\int_{\rbb_+} x F(\D x)\Big)^{p} =
\int_{\rbb_+} x^{p} F(\D x).
   \end{align}
Then $F$ is a spectral measure.
   \end{lemma}
   \begin{proof}
First we consider the case when $p
> 1$. By Naimark's dilation theorem (see
\cite[Theorem~6.4]{Ml78}), there exist
a Hilbert space $\kk$ containing $\hh$
and a spectral measure $E\colon
\borel(\rbb_+) \to \ogr(\kk)$ such~that
   \begin{gather} \label{Naim-1}
F(\varDelta) = P E(\varDelta)|_{\hh},
\quad \varDelta \in \borel(\rbb_+),
   \\  \label{Naim-2}
\text{$\kk$ is the only closed vector
subspace of $\kk$ reducing $E$ and
containing $\hh$,}
   \end{gather}
where $P\in \ogr(\kk)$ is the
orthogonal projection of $\kk$ onto
$\hh$. It suffices to show that
$\kk=\hh$, because then by
\eqref{Naim-1}, $F=E$. Suppose to the
contrary that $P\neq I_{\kk}$. It
follows from \eqref{Naim-1} and
\eqref{Naim-2} that the closed supports
of $E$ and $F$ coincide (see e.g., the
proof of \cite[Theorem~4.4]{Ja02}), so
$E$ has compact support. Therefore
$T:=\int_{\rbb_+} x F(\D x) \in
\ogr(\hh)$ and $S:=\int_{\rbb_+} x E(\D
x) \in \ogr(\kk)$, and both operators
$T$ and $S$ are positive. Using the
Stone-von Neumann functional calculus,
we get
   \begin{align*}
\langle T^j f, g \rangle &
\overset{\eqref{form-ua}\&
\eqref{kap-lok} }= \int_{\rbb_+} x^j
\langle F(\D x)f, g \rangle
   \\
&\hspace{2.6ex}\overset{\eqref{Naim-1}}=
\int_{\rbb_+} x^j \langle E(\D x)f, g
\rangle = \langle S^j f, g \rangle,
\quad f,g \in \hh, \, j=1,p,
   \end{align*}
which implies that
   \begin{align} \label{kon-1}
T^j = P S^j|_{\hh}, \quad j=1,p.
   \end{align}
Let $\tilde T \in \ogr(\kk)$ be defined
by $\tilde T= T \oplus 0$. It follows
from \eqref{kon-1} that
   \begin{align*}
\tilde T^j = P S^j P, \quad j=1,p.
   \end{align*}
Combined with the Stone-von Neumann
functional calculus and
\eqref{kap-lok}, this yields
   \begin{align*}
Pf(S^{p})P=f(PS^{p}P),
   \end{align*}
where $f$ is as in the first proof of
Theorem~\ref{maintw}, that is
$f(x)=x^{\frac{1}{p}}$ for $x\in
\rbb_+$ (note that $0 < \frac{1}{p} <
1$). Using Theorem~\ref{hans}, we
deduce that $P$ commutes with $S^{p}$
and thus by Theorem~\ref{rsn}, $\hh$
reduces $S$. Hence, $\hh$ reduces the
spectral measure $E$ (see
\cite[Proposition~5.15]{Sch12}). By
\eqref{Naim-2}, $\kk=\hh$, which gives
a contradiction. This proves the
conclusion of the lemma for $p > 1$.

We now consider the case when $p < 1$.
Given $\alpha \in (0,\infty)$, we
define the function $\psi_{\alpha}
\colon \rbb_+\to \rbb_+$ by
   \begin{align}  \label{psi-al}
\psi_{\alpha}(x)=x^{\alpha}, \quad x\in
\rbb_+.
   \end{align}
Using \eqref{form-ua}, the measure
transport theorem and the Stone-von
Neumann functional calculus, one can
deduce from \eqref{kap-lok} that
   \begin{align*}
\Big(\int_{\rbb_+} x \, (F\circ
\psi_{p}^{-1}) (\D x)\Big)^{1/p} =
\int_{\rbb_+} x^{1/p} (F\circ
\psi_{p}^{-1}) (\D x),
   \end{align*}
where $F\circ \psi_{p}^{-1}$ is defined
as in \eqref{num-er} with $p$ in place
of $n$. Therefore, by the previous
paragraph, $F\circ \psi_{p}^{-1}$ is a
spectral measure and consequently so is
$F$.
   \end{proof}
   We are now ready to provide the
promised proof.
   \begin{proof}[Proof of Theorem~\ref{maintw4}]
(iv)$\Rightarrow$(i) By \eqref{form-ua}
and the measure transport theorem, we
have
   \begin{align*}
T^p = \int_{\rbb_+} x^p F(\D x) =
\int_{\rbb_+} x^{p/\alpha} (F\circ
\psi_{\alpha}^{-1})(\D x), \quad
p=\alpha, \beta,
   \end{align*}
where $\psi_{\alpha}$ is as in
\eqref{psi-al}. Applying the Stone-von
Neumann functional calculus yields
   \begin{align*}
\Big(\int_{\rbb_+} x \, (F\circ
\psi_{\alpha}^{-1})(\D x)
\Big)^{\beta/\alpha} = \int_{\rbb_+}
x^{\beta/\alpha} \, (F\circ
\psi_{\alpha}^{-1})(\D x).
   \end{align*}
Using Lemma~\ref{maintw3} with
$p=\beta/\alpha$ implies that $F$ is a
spectral measure.

The implications (i)$\Rightarrow$(ii)
and (i)$\Rightarrow$(iii) are immediate
from the Stone-von Neumann functional
calculus. The implications
(iii)$\Rightarrow$(ii) and
(iii)$\Rightarrow$(iv) are obvious.
Finally, the implication
(ii)$\Rightarrow$(i) is a direct
consequence of (iv)$\Rightarrow$(i).
   \end{proof}
   \subsection*{Acknowledgements}
The authors would like to thank Professor
M. H. Mortad for reminding them of
Problem~\ref{prob}.
   \bibliographystyle{amsalpha}

\begin{thebibliography}{99}

\bibitem{Ag85}
J. Agler, Hypercontractions and
subnormality, {\em J. Operator Theory}
{\bf 13} (1985), 203--217.

\bibitem{Ash00}
R. B. Ash, {\em Probability and measure
theory}, Harcourt/Academic Press,
Burlington, 2000.

\bibitem{B-C-R}
C. Berg, J. P. R. Christensen, P. Ressel,
{\em Harmonic analysis on semigroups},
Springer-Verlag, Berlin, 1984.

\bibitem{Bha97}
R. Bhatia, {\em Matrix analysis},
Graduate Texts in Mathematics, 169,
Springer-Verlag, New York, 1997.

\bibitem{brow53}
A. Brown, On a class of operators, {\em
Proc. Amer. Math. Soc.} {\bf 4} (1953),
723--728.

\bibitem{BJJS18}
P. Budzy\'{n}ski, Z. J. Jab{\l}o\'{n}ski,
I. B. Jung, J. Stochel, Unbounded
weighted composition operators in
$L^2$-spaces, {\em Lect. Notes Math.},
2209, Springer 2018.

\bibitem{con91}
J. B. Conway, {\em The theory of
subnormal operators}, Math. Surveys
Monographs, Amer. Math. Soc., Providence,
1991.

\bibitem{con87}
J. B. Conway, B. B. Morrel, Roots and
logarithms of bounded operators on
Hilbert space, {\em J. Funct. Anal.} {\bf
70} (1987), 171--193.

\bibitem{Curto05}
R. E. Curto, S. H. Lee, J. Yoon,
$k$-hyponormality of multivariable
weighted shifts, {\em J. Funct. Anal.}
{\bf 229} (2005), 462--480.

\bibitem{Curto07}
R. E. Curto, S. H. Lee, J. Yoon,
Hyponormality and subnormality for powers
of commuting pairs of subnormal
operators, {\em J. Funct. Anal.} {\bf
245} (2007), 390--412.

\bibitem{Curto20}
R. E. Curto, S. H. Lee, J. Yoon,
Quasinormality of powers of commuting
pairs of bounded operators, {\em J.
Funct. Anal.} {\bf 278} (2020), 108342.

\bibitem{Dono74}
W. F. Donoghue, {\em Monotone matrix
functions and analytic continuation},
Springer-Verlag, Berlin, 1974.

\bibitem{Embry73}
M. R. Embry, A generalization of the
Halmos-Bram criterion for subnormality,
{\em Acta Sci. Math. $($Szeged$)$} {\bf
35} (1973), 61--64.

\bibitem{Fur01} T. Furuta, {\em Invitation to
linear operators}, Taylor \& Francis,
Ltd., London, 2001.

\bibitem{hal50}
P. R. Halmos, Normal dilations and
extensions of operators, {\em Summa
Brasil. Math.} {\bf 2} (1950), 125--134.

\bibitem{hal82}
P. R. Halmos, {\em A Hilbert space
problem book}, Springer-Verlag, New York
Inc., 1982.

\bibitem{hal53}
P. R. Halmos, G. Lumer, J. J. Schaffer,
Square roots of operators, {\em Proc.
Amer. Math. Soc.} {\bf 4} (1953),
142--149.

\bibitem{hal54}
P. R. Halmos, G. Lumer, Square roots of
operators II, {\em Proc. Amer. Math.
Soc.} {\bf 5} (1954), 589--595.

\bibitem{Han80}
F. Hansen, An operator inequality, {\em
Math. Ann.} {\bf 246} (1980), 249--250.

\bibitem{Han13}
F. Hansen, The fast track to L\"{o}wner's
theorem, {\em Linear Algebra Appl.} {\bf
438} (2013), 4557--4571.

\bibitem{He51}
E. Heinz, Beitr\"age zur
St\"orungstheorie der Spektralzerlegung,
{\em Math. Ann.} {\bf 123} (1951),
415--438.

\bibitem{Ja02}
Z. J. Jab{\l}o\'{n}ski, Complete
hyperexpansivity, subnormality and
inverted boundedness conditions, {\it
Integr. Equ. Oper. Theory} {\bf 44}
(2002), 316--336.

\bibitem{jabl14}
Z. J. Jab{\l}o\'{n}ski, I. B. Jung, J.
Stochel, Unbounded quasinormal operators
revisited, {\em Integr. Equ. Oper.
Theory} {\bf 79} (2014), 135--149.

\bibitem{Jib10}
A. A. S. Jibril, On operators for which
$T^{*2} T^2 = (T^*T)^2$, {\em Int. Math.
Forum}, {\bf 46} (2010), 2255--2262.

\bibitem{Ju-St08}
I. B. Jung, J. Stochel, Subnormal
operators whose adjoints have rich point
spectrum, {\em J. Funct. Anal.} {\bf 255}
(2008), 1797--1816.

\bibitem{Lam76}
A. Lambert, Subnormality and weighted
shifts, {\em J. London Math. Soc.} {\bf
14} (1976), 476--480.

\bibitem{Lo34}
K. L\"owner, \"Uber monotone
Matrixfunktionen, {\em Math. Z.} {\bf 38}
(1934), 177--216.

\bibitem{Ml78}
W. Mlak, {\em Dilations of Hilbert
space operators $($general theory$)$},
Dissertationes Math. {\bf 153} (1978),
61p.

\bibitem{pp16}
P. Pietrzycki, The single equality $A^{
*n}A^n=(A^* A)^n$ does not imply the
quasinormality of weighted shifts on
rootless directed trees, {\em J. Math.
Anal. Appl} {\bf 435} (2016), 338--348.

\bibitem{pp18}
P. Pietrzycki, Reduced commutativity of
moduli of operators, {\em Linear Algebra
Appl.} {\bf 557} (2018), 375--402.

\bibitem{put57}
C. R. Putnam, On square roots of normal
operators, {\em Proc. Amer. Math. Soc.}
{\bf 8} (1957), 768--769.

\bibitem{Rud73}
W. Rudin, {\em Functional analysis},
McGraw-Hill Series in Higher Math.,
McGraw-Hill Book Co., New York, 1973.

\bibitem{Sch12}
K. Schm\"{u}dgen, {\em Unbounded
self-adjoint operators on Hilbert space,}
Graduate Texts in Mathematics, 265,
Springer, Dordrecht, 2012.

\bibitem{Sim19}
B. Simon, {\em Loewner's theorem on
monotone matrix functions}, Grundlehren
Math. Wissen. 354, Springer, 2019.

\bibitem{sta62}
J. G. Stampfli, Hyponormal operators,
{\em Pacific J. Math.}, {\bf 12} (1962),
1453--1458.

\bibitem{sta66}
J. G. Stampfli, Which weighted shifts are
subnormal?, {\em Pacific J. Math.}, {\bf
17} (1966), 367--379.

\bibitem{Sto92}
J. Stochel, Decomposition and
disintegration of positive definite
kernels on convex $*$-semigroups, {\em
Ann. Polon. Math.} {\bf 56} (1992),
243--294.

\bibitem{uch93}
M. Uchiyama, Operators which have
commutative polar decompositions, {\em
Oper. Theory Adv. Appl.} {\bf 62} (1993),
197--208.

\bibitem{uch01}
M. Uchiyama, Inequalities for semibounded
operators and their applications to
log-hyponormal operators, {\em Oper.
Theory Adv. Appl.}, {\bf 127} (2001),
599--611.

\bibitem{Weid80}
J. Weidmann, {\it Linear operators in
Hilbert spaces}, Springer-Verlag, Berlin,
Heidelberg, New York, 1980.

\bibitem{wog85}
W. Wogen, Subnormal roots of subnormal
operators, {\em Integr. Equ. Oper.
Theory} {\bf 8} (1985), 432--436.
   \end{thebibliography}
   
   \end{document}